\documentclass[12pt,a4paper]{article}
\usepackage{amssymb, amscd, amsthm, amsmath,latexsym, amstext,mathrsfs}
\usepackage[all]{xy}
\usepackage[dvips]{graphicx}

\numberwithin{equation}{section}

\newtheorem{thm}{Theorem}[section]

\newtheorem{lemma}[thm]{Lemma}
\newtheorem{prop}[thm]{Proposition}

\theoremstyle{remark}
\newtheorem{remark}[thm]{\textbf{Remark}}

\theoremstyle{definition}

\newcommand{\Br}{\mathrm{Br}}
\newcommand{\Spec}{\mathrm{Spec}\,}
\newcommand{\red}{\mathrm{red}}
\newcommand{\wh}[1]{\widehat{#1}}
\newcommand{\ov}[1]{\overline{#1}}
\newcommand{\set}[1]{\{\,{#1}\,\}}

\begin{document}

\title{\textbf{Local-global principle for quadratic forms over fraction fields of two-dimensional henselian domains}}
\author{Yong HU\footnote{Math\'{e}matiques, B\^{a}timent 425, Universit\'{e} Paris-Sud, 91405,
Orsay Cedex,  France,\ \ \ e-mail: yong.hu2@math.u-psud.fr}}

\maketitle

\begin{abstract}
Let $R$ be a 2-dimensional normal excellent henselian local domain
in which $2$ is invertible and let $L$ and $k$ be respectively its
fraction field and residue field. Let $\Omega_R$ be the set of rank
1 discrete valuations of $L$ corresponding to codimension 1 points
of regular proper models of $\Spec R$.  We prove that a quadratic
form $q$ over $L$ satisfies the local-global principle with respect
to $\Omega_R$ in the following two cases: (1) $q$ has rank 3 or 4;
(2) $q$ has rank $\ge 5$ and $R=A[\![y]\!]$, where $A$ is a complete
discrete valuation ring with a not too restrictive condition on the
residue field $k$, which is satisfied when $k$ is $C_1$.
\end{abstract}


\section{Statements of results}

Let $R$ be a 2-dimensional excellent henselian local domain and let
$L$ and $k$ be respectively its fraction field and residue field.
Assume that the characteristic of $k$ is not $2$.

Colliot-Th\'el\`ene, Ojanguren and Parimala \cite{CTOP} proved that
any quadratic form of rank at least $5$ over $L$ is isotropic when
$k$ is separably closed, and that the local-global principle with
respect to all discrete valuations (of rank 1) on $L$ holds for
quadratic forms of rank $3$ or $4$ when $k$ is separably closed or
finite. For the first result, the special case where
$R=\mathbb{C}[\![x,\,y]\!]$ was proven earlier in \cite{CDLR} using
the Weierstra{\ss} preparation theorem. On the other hand, Jaworski
\cite{Ja} proved that if $k$ is an algebraically closed field, then
quadratic forms of any rank over $L=k(\!(x\,,\,y)\!)$  satisfy the
local-global principle with respect to all discrete valuations on
$L$.

In the case where $k$ is finite, however, whether the local-global
principle holds for quadratic forms of rank $\ge 5$ is left an open
question. In this paper, we give an affirmative answer to this
question in the case where $R=A[\![y]\!]$ is the ring of formal
power series in one variable over a complete discrete valuation ring
$A$. Also, we prove that the result of Colliot-Th\'el\`ene,
Ojanguren and Parimala about the local-global principle for
quadratic forms of rank 3 or 4 is still valid without the assumption
that $k$ is separably closed or finite.

\

The more precise statements are the following.

\begin{thm}\label{thm1p1}
Let $R$ be a $2$-dimensional normal excellent henselian local domain
in which $2$ is invertible. Let $L$ and $k$ be respectively the
fraction field and the residue field of $R$. For any regular
integral scheme $\mathcal{M}$ equipped with a proper birational
morphism $\mathcal{M}\to\Spec R$, let $\Omega_{\mathcal{M}}$ denote
the set of  rank $1$ discrete valuations of $L$ that correspond to
codimension $1$ points of $\mathcal{M}$. Let $\Omega_R$ be the union
of all $\Omega_{\mathcal{M}}$.

Then the local-global principle with respect to $\Omega_R$ holds for
quadratic forms of rank $3$ or $4$ over $L$. Namely, if a quadratic
form of rank $3$ or $4$ over $L$ has a nontrivial zero over the
$w$-adic completion $L_w$ for every $w\in\Omega_R$, then it has a
nontrivial zero over $L$.
\end{thm}

\begin{thm}\label{thm1p2}
Let $A$ be a complete discrete valuation ring in which $2$ is
invertible, and let $K$ and $k$ be respectively its fraction field
and residue field. Let $R=A[\![y]\!]$ and $L=\mathrm{Frac}(R)$ the
fraction field of $R$. Define $\Omega_R$ as in Thm.$\;\ref{thm1p1}$.

Assume that the residue field $k$ has the following property:

$(*)$ for every finite field extension $k'/k$, every quadratic form
of rank $\ge 3$ over $k'$  is isotropic.

Then the local-global principle with respect to discrete valuations
in $\Omega_R$ holds for quadratic forms of rank $\ge 5$ over $L$.
\end{thm}

Recall that a field $k$ is called a $C_i$ field if every homogeneous
polynomial of degree $d$ in $n>d^i$ variables has a nontrivial zero
over $k$. A finite field extension of a $C_i$ field is again a $C_i$
field. Clearly, a $C_1$ field $k$ has property $(*)$. So as typical
examples to which Thm.$\;$\ref{thm1p2} applies, we may take
$R=\mathbb{F}[\![x\,,\,y]\!]$ where $\mathbb{F}$ is a finite field
of characteristic $>2$, or $R=\mathcal{O}_K[\![y]\!]$ where
$\mathcal{O}_K$ is the ring of integers of a $p$-adic number field
$K$ ($p$ is an odd prime).

\begin{remark}\label{remark1p3}
Note that property $(*)$ implies the following:

$(**)$ for every finite field extension $K'/K$, every quadratic form
of rank $\ge 5$ over $K'$  is isotropic.

Indeed, the integral closure $A'$ of $A$ in $K'$ is a complete
discrete valuation ring and is finite over $A$ (cf. \cite[p.28,
$\S$II.2, Prop.$\;$3]{Ser1}). The residue field $k'$ of $A'$ is a
finite extension of $k$. Any quadratic form $q$ over $K'$ is
isometric to a form $q_1\bot\, t.q_2$, where $t$ is a uniformizer of
$A'$ and the coefficients of $q_1,\,q_2$ are all units in $A'$. When
$q$ has rank $\ge 5$ and $k$ has property $(*)$, a standard argument
using Springer's lemma (cf. Lemma$\;$\ref{lemma4p1}) shows that $q$
is isotropic over $K'$.
\end{remark}

\

Let $A\,,\,k\,,\,K$ and so on be as in Thm.$\;\ref{thm1p2}$. Let
$x\in A$ be a uniformizer of $A$ and $F=K(y)$ the function field of
$\mathbb{P}^1_K$. For any regular integral scheme $\mathcal{P}$
equipped with a proper flat morphism $\mathcal{P}\to \Spec A$ with
generic fiber $\mathcal{P}\times_AK\cong \mathbb{P}^1_K$, let
$\Omega_{\mathcal{P}}$ denote the set of rank $1$ discrete
valuations of $F$ that correspond to codimension $1$ points of
$\mathcal{P}$. Let $\Omega_A$ be the union of all
$\Omega_{\mathcal{P}}$. Then we have the following proposition.

\begin{prop}\label{prop1p4}
With notation as above, let $q/F=\, \langle a_1\,,\,\dotsc,\, a_r
\rangle$ be a nonsingular diagonal quadratic form of rank $r\ge 5$
with $a_i\in A[y]$. Let $\Sigma\subseteq A$ be a fixed set of
representatives of $k^*$ in $A$. Assume that
\begin{equation}\label{eq1p1}
a_i=\lambda_i.\,x^{n_i}.\,P_i\,, \end{equation}where $\lambda_i\in
\Sigma\,,\,n_i\in\set{0\,,\,1}$ and $P_i$ is a distinguished
polynomial of degree $m_i$ in $A[y]\,($meaning that $P_i$ is a monic
polynomial in $A[y]$ whose reduction mod $x$ is $y^{m_i}\in
k[y]\,)$.

If for every $w\in\Omega_R$, $q$ is isotropic over the completion
$L_w$ of $L$ with respect to $w$, then for every $v\in\Omega_A$, $q$
is isotropic over the completed field $F_v$.
\end{prop}

As we shall see at the end of the paper, Thm.$\;$\ref{thm1p2}
follows by combining the above proposition with a theorem of
Colliot-Th\'el\`ene, Parimala and Suresh \cite{CTPaSu} on quadratic
forms over $F=K(y)$, whose proof builds upon earlier work of
Harbater, Hartmann and Krashen \cite{HHK}.

\section{Valuations coming from blow-ups}

\begin{lemma}\label{lemma2p1}
Let $A$ be an excellent local domain with residue field $k$ and
$\mathcal{X}$ an integral $A$-scheme of finite type. Let $F$ be the
function field of $\mathcal{X}$ and $v$ a rank $1$ discrete
valuation of $F$ with valuation ring $\mathcal{O}_v$. Assume that
$v$ is centered on $\mathcal{X}$ at a point $x$ in the closed fiber
$X:=\mathcal{X}\times_Ak$ and that the residue field $\kappa(v)$ of
$\mathcal{O}_v$ has transcendence degree
$\mathrm{trdeg}_k\kappa(v)=\dim \mathcal{X}-1$ over $k$. Let
$\mathcal{Y}=\Spec\mathcal{O}_v$ and $y\in \mathcal{Y}$ the closed
point of $\mathcal{Y}$. Let $f:\mathcal{Y}\to \mathcal{X}$ be the
natural morphism induced by the inclusion
$\mathscr{O}_{\mathcal{X},\,x}\subseteq \mathcal{O}_v$. Define
schemes $\mathcal{X}_n\,,n\in\mathbb{N}$ and morphisms $f_n:
\mathcal{Y}\to \mathcal{X}_n\,,\,n\in \mathbb{N}$ as follows:

Set $\mathcal{X}_0=\mathcal{X}$ and $f_0=f$. When $f_i:
\mathcal{Y}\to \mathcal{X}_i$ is already defined, let
$\mathcal{X}_{i+1}\to \mathcal{X}_i$ be the blow-up of
$\mathcal{X}_i$ along the closure of $x_i:=f_i(y)$ and let
$f_{i+1}:\mathcal{Y}\to \mathcal{X}_{i+1}$ be the induced morphism.

Then for some large enough $n$, the morphism $f_n: \mathcal{Y}\to
\mathcal{X}_n$ induces an isomorphism
$\mathscr{O}_{\mathcal{X}_n\,,\,x_n}\cong \mathcal{O}_v$.
\end{lemma}
\begin{proof}
The following proof is an easy adaptation of the proof of the
geometric case, as given in \cite[p.61, Lemma$\;$2.45]{KM}.

Let $\mathscr{O}_n:=\mathscr{O}_{\mathcal{X}_n,\,x_n}$. The ring
theoretic construction of $\mathscr{O}_n$ is as follows. Assume that
$\mathscr{O}_n$ (with maximal ideal $\mathfrak{m}_n$) is already
defined. Pick a system of generators $z_1,\,\dotsc, z_r$ of
$\mathfrak{m}_n$ such that $v(z_1)\le \cdots\le v(z_r)$. Let
$\mathscr{O}'_n=\mathscr{O}_n[z_2/z_1\,,\dotsc, \,z_r/z_1]$. Then
$\mathscr{O}_{n+1}$ is the localization of  $\mathscr{O}'_{n}$ at
$\mathscr{O}'_n\cap\mathfrak{m}_v$.

The same argument as in the proof of \cite[p.61, Lemma$\;$2.45]{KM}
applies here and shows that $\mathcal{O}_v=\bigcup_{n\ge
0}\mathscr{O}_n$. Pick elements $u_1,\dotsc,
u_t\in\mathcal{O}_v\subseteq F$ such that the reductions $\ov{u}_i$
form a transcendence basis of
$\kappa(v)=\mathcal{O}_v/\mathfrak{m}_v$ over $k$. Choose $n$ big
enough so that $u_1,\dotsc, u_t\in\mathscr{O}_n$. Then
$\kappa(v)=\mathcal{O}_v/\mathfrak{m}_v$ is an algebraic extension
of $\kappa(x_n)=\mathscr{O}_n/\mathfrak{m}_n$ and
\[
\mathrm{trdeg}_k\kappa(x_n)=\mathrm{trdeg}_k\kappa(v)=\dim\mathcal{X}-1\,.
\]
The closure $Z_n:=\ov{\set{x_n}}$ of $x_n$ in $\mathcal{X}_n$ is an
algebraic scheme over $k$. So we have
\[
\dim Z_n=\mathrm{trdeg}_k\kappa(x_n)=\dim \mathcal{X}-1\,.
\]
 By \cite[p.334, Coro.$\;$8.2.7]{Liu}, we have $\dim\mathcal{X}_n=\dim\mathcal{X}$. Hence,
\[
\dim\mathscr{O}_n=\mathrm{codim}(Z_n\,,\,\mathcal{X}_n)\le
\dim\mathcal{X}_n-\dim Z_n=1\,.
\]But $\mathscr{O}_n\subseteq\mathcal{O}_v$ and the discrete valuation ring  $\mathcal{O}_v$ is unequal to its fraction field
$F=\mathrm{Frac}(\mathcal{O}_v)=\mathrm{Frac}(\mathscr{O}_n)$, so
$\dim \mathscr{O}_n=1$. Let $R'\subseteq F$ be the normalization of
$\mathscr{O}_n$ and let $\mathfrak{m}'=\mathfrak{m}_v\cap R$. Then
$R'$ is a Dedekind domain and $R'_{\mathfrak{m}'}$ is a discrete
valuation ring contained in $\mathcal{O}_v$ with fraction field $F$.
Therefore, $R'_{\mathfrak{m}'}=\mathcal{O}_v$. The ring
$\mathscr{O}_n$ is a Nagata ring (see e.g. \cite[p.340,
Prop.$\;$8.2.29 and p.343, Thm.$\;$8.2.39]{Liu}). So $R'$ is a
finitely generated $\mathscr{O}_n$-module. Thus we have
$R'\subseteq\mathscr{O}_N$ for some large $N\in\mathbb{N}$. Then it
follows that $\mathcal{O}_v=\mathscr{O}_{N+1}$. The lemma is thus
proved.
\end{proof}

\section{Proof of Theorem$\;\ref{thm1p1}$}

Thm.$\;$\ref{thm1p1} is a statement generalizing
\cite[Thm.$\;$3.1]{CTOP}, where the result is only established under
the hypothesis that $k$ is separably closed or finite. In our proof
the observation that \cite[Prop.$\;$1.14]{CTOP} holds over an
arbitrary field $k$ is the key point which makes it possible to get
rid of this restriction on $k$. In addition, Lemma$\;$\ref{lemma2p1}
will be used in order to obtain the local-global principle for
valuations in the subset $\Omega_R$ instead of the set of all
discrete valuations.

\begin{lemma}\label{lemma3p1}
Let $R$ be a two-dimensional normal excellent henselian local domain
with fraction field $L$, $L'/L$ a finite field extension and $R'$
the integral closure of $R$ in $L'$. Let $w'$ be a discrete
valuation of $L'$ lying over a discrete valuation $w$ of $L$.

If $w'$ corresponds to a codimension $1$ point on a regular proper
model $\mathcal{X}'$ of $R'\,($i.e., $\mathcal{X}'$ is a regular
integral scheme equipped with a proper birational morphism
$\mathcal{X}'\to\Spec R')$, then $w$ corresponds to a codimension
$1$ point on a regular proper model $\mathcal{X}$ of $R$.
\end{lemma}
\begin{proof}
Let $k$ (resp. $k'$) be the residue field of $R$ (resp. $R'$). Since
$R$ is excellent, $R'$ is finite over $R$ and hence $k'/k$ is a
finite extension. Let $x'\in \mathcal{X}'$ be the center of $w'$ on
$\mathcal{X}'$, $p'$ the canonical image of $x'$ in $\Spec R'$ and
$p$ the canonical image of $p'$ in $\Spec R$.

If $p$ is not the closed point of $\Spec R$, then it has codimension
1 in $\Spec R$ and the valuation ring $\mathcal{O}_w$ of $w$ is
equal to the local ring of $p$ in $\Spec R$, since $R$ is a
2-dimensional normal local domain. Let $V$ be the complement of the
closed point in $\Spec R$. For any regular proper model $\pi:
\mathcal{X}\to \Spec R$, which exists by resolution of
singularities, $\pi^{-1}(V)\to V$ is an isomorphism since $R$ is
normal (cf. \cite[p.150, Coro.$\;$4.4.3]{Liu}). Hence, the point
$x=\pi^{-1}(p)$  has codimension 1 in $\mathcal{X}$ and is the
center of $w$ on $\mathcal{X}$.

Now assume that $p$ is the closed point of $\Spec R$. Then
$x'\in\mathcal{X}'$ lies in the closed fiber of $\mathcal{X}'/R'$
and is the generic point of an integral curve over $k'=\kappa(p')$.
Hence, the residue field $\kappa(w')$ of $w'$ has transcendence
degree 1 over $k'$. Since $k'/k$ and $\kappa(w')/\kappa(w)$ are
finite extensions, this implies that the residue field $\kappa(w)$
has transcendence degree 1 over $k$. By taking any regular proper
model $\mathcal{X}\to \Spec R$ and applying Lemma$\;$\ref{lemma2p1}
to the ring $R$ and the $R$-scheme $\mathcal{X}$, we conclude that
there is a morphism $\mathcal{X}_n\to \mathcal{X}$ obtained by a
sequence of blow-ups such that the center of $w$ on $\mathcal{X}_n$
is a
 point of codimension 1, which completes the proof.
\end{proof}

Given a scheme $Y$, we will denote by
$\Br(Y)=H^2_{\text{\'et}}(Y\,,\,\mathbb{G}_m)$ its cohomological
Brauer group.

\begin{proof}[Proof of Thm.$\;\ref{thm1p1}$]
For any $a\,,\,b\in L^*$, the isotropy of the rank $3$ form $\langle
1\,,\,a\,,\,b\rangle$ is equivalent to the isotropy of the rank $4$
form $\langle 1\,,\,a\,,\,b\,,\,ab\rangle$. So we may restrict to
the case of rank $4$ forms. Let $q$ be a rank $4$ quadratic form
over $L$ which is isotropic over $L_w$ for every $w\in\Omega_R$.
After scaling we may assume without loss of generality that
$q=\langle 1\,,\,a\,,\,b\,,\,abd\rangle$ with $a,\,b\,,\,d\in L^*$.

First assume that $d$ is a square in $L$. Then the quadratic form
$q$ is isomorphic to the norm form of a quaternion algebra, whose
class in the Brauer group $\Br(L)$ will be denoted by $\alpha$. The
form $q$ is isotropic if and only if $\alpha=0$ in the Brauer group.

Take a proper birational morphism $\mathcal{X}\to\Spec R$ with
$\mathcal{X}$ a regular integral scheme such that the closed fiber
$X$ of $\mathcal{X}/R$ is a curve over $k$. For each $w\in\Omega_R$
corresponding to a codimension 1 point of $\mathcal{X}$, the
canonical image $\alpha_w$ of $\alpha$ in $\Br(L_w)$ is trivial
since $q$ is isotropic over $L_w$ by assumption. In particular, the
residue of $\alpha$ at every codimension 1 point of $\mathcal{X}$ is
trivial. Since $\mathcal{X}$ is a regular integral scheme, it
follows that $\alpha\in\Br(L)$ lies in the subgroup
$\Br(\mathcal{X})$. By \cite[Thm.$\;$1.8 (c) and
Lemma$\;$1.6]{CTOP}, we have canonical isomorphisms
$\Br(\mathcal{X})\cong\Br(X)\cong\Br(X_{\red})$. Identify
$\alpha\in\Br(\mathcal{X})$ with its canonical image in
$\Br(X_{\red})$. We will apply \cite[Prop.$\;$1.14]{CTOP} to show
that $\alpha=0$.

Let $f: Z\to X_{\red}$ be the normalization of the reduced curve
$X_{\red}/k$ and let $D\subseteq X_{\red}$ be the closed subscheme
defined by the conductor of $f$. Then \cite[Prop.$\;$1.14]{CTOP}
says that the natural map $\Br(X_{\red})\to\Br(Z)\times\Br(D)$ is
injective. Let $(\alpha_1\,,\,\alpha_2)\in \Br(Z)\times \Br(D)$ be
the image of $\alpha\in\Br(X_{\red})$. Each reduced irreducible
component $T$ of $Z$ is a regular integral curve whose function
field $k(T)$ is the residue field $\kappa(w)$ of a codimension 1
point $w$ of the 2-dimensional regular scheme $\mathcal{X}$. Since
$\alpha$ vanishes in $\Br(L_w)$ by hypothesis, the specialisation of
$\alpha$ in $\Br(\kappa(w))=\Br(k(T))$ is zero. The natural map
$\Br(T)\to \Br(k(T))$ is an injection for the regular scheme $T$, so
the canonical image of $\alpha$ in $\Br(T)$ is zero. Since this
holds for every irreducible component $T$ of $Z$, we have
$\alpha_1=0$ in $\Br(Z)$.

To show that $\alpha_2=0$ in $\Br(D)$, it suffices to prove that
$\alpha_2$ vanishes at each closed point $x$ of $X_{\red}$, by a
0-dimensional variant of \cite[Lemma$\;$1.6]{CTOP}. The point $x$ is
also a closed point of $\mathcal{X}$. We may choose a 1-dimensional
closed integral subscheme $C$ of $\mathcal{X}$ which contains $x$ as
a regular point and let $\omega\in\mathcal{X}$ be the generic point
of $C$. Our hypothesis implies that $\alpha\in\Br(\mathcal{X})$
vanishes at $\omega$, and it follows that there is a regular open
subscheme $U$ of $C$, containing $x$, such that $\alpha|_U=0$ in
$\Br(U)\subseteq \Br(\kappa(\omega))$. Hence,
$\alpha_2(x)=\alpha(x)=0$. We have thus proved that $\alpha=0$ in
$\Br(L)$, whence the isotropy of the rank $4$ quadratic form
$q=\langle 1\,,\,a\,,\,b\,,\,abd\rangle$.

Now suppose that $d$ is not a square in $L$. Let $L'=L(\sqrt{d})$
and $R'$ the integral closure of $R$ in $L'$. Then $R'$ and $L'$
satisfy the same assumptions as $R$ and $L$. Let $w'$ be a discrete
valuation on $L'$ corresponding to a codimension $1$ point of a
regular proper model $\mathcal{X}'/R'$. By Lemma$\;$\ref{lemma3p1},
$w'$ lies over a discrete valuation $w$ in $\Omega_R$. The isotropy
of $q$ over $L_w$ implies the isotropy of $q_{L'}$ over $L'_{w'}$.

Thus the quadratic form $q_{L'}$ over $L'$ has trivial determinant
and is isotropic over $L'_{w'}$ for every $w'\in\Omega_{R'}$, where
the set $\Omega_{R'}$ of discrete valuations of $L'$ is defined in
the same way as $\Omega_R$. By the previous case, $q_{L'}$ is
isotropic over $L'$. By \cite[p.197, Chapt.$\;$VII,
Thm.$\;$3.1]{Lam}, either $q$ is isotropic over $L$ or $q$ contains
a multiple of $\langle 1\,,\,-d\rangle$. In the latter case, since
$\det(q)=d\mod{(L^*)^2}$, $q$ also contains a rank 2 form of
determinant $-1$. Hence $q$ is isotropic over $L$, which completes
the proof.
\end{proof}

\section{Valuations centered on the special fiber}

Most of the present section and the next will be devoted to the
proof of Prop.$\;$\ref{prop1p4}. The lemma below will be used
frequently and referred to as Springer's lemma in what follows.

\begin{lemma}[Springer's lemma, {\cite[p.148,
Prop.$\;$VI.1.9]{Lam}}]\label{lemma4p1} Let $A$ be a complete
discrete valuation ring in which $2$ is invertible. Let $K$ and $k$
be respectively its fraction field and residue field. Let
$\alpha_1,\dotsc, \alpha_r$ and $\beta_1,\dotsc,\beta_s$ be units of
$A$ and let $\ov{\alpha}_i\in k$ and $\ov{\beta}_j\in k$ be their
residue classes. Let $\pi$ be a uniformizer of $A$.

Then the quadratic form $\langle \alpha_1\,,\dotsc,
\alpha_r\rangle\bot\,\pi.\langle \beta_1\,,\dotsc, \beta_s\rangle$
over $K$ is anisotropic if and only if the two residue forms
\[
\langle\ov{\alpha}_1\,,\dotsc, \ov{\alpha}_r\rangle\quad\text{and
}\quad \langle \ov{\beta}_1\,,\dotsc,\ov{\beta}_s\rangle
\]are both anisotropic over $k$.
\end{lemma}

We shall now start the proof of Prop.$\;$\ref{prop1p4}. Recall that
$\Omega_A$ is the union of all $\Omega_{\mathcal{P}}$, where
$\mathcal{P}$ is a regular integral proper flat $A$-scheme with
generic fiber $\mathcal{P}\times_AK\cong\mathbb{P}^1_K$ and
$\Omega_{\mathcal{P}}$ is the set of rank 1 discrete valuations on
$F=K(y)$ that correspond to codimension 1 points of $\mathcal{P}$.
We will fix a discrete valuation $v\in \Omega_A$ and let
$\mathcal{O}_v\subseteq F$ denote the valuation ring of $v$,
$\pi_v\in\mathcal{O}_v$ a uniformizer of $v$,
$\mathfrak{m}_v=\pi_v\mathcal{O}_v$ and $\kappa(v)$ the residue
field of $\mathcal{O}_v$. The $v$-adic completion of
$\mathcal{O}_v\subseteq F$ will be written as
$\wh{\mathcal{O}}_v\subseteq F_v$. If $w$ is a discrete valuation of
$L$, similar notations like $\mathcal{O}_w$, $\mathfrak{m}_w$,
$\kappa(w)$, $\wh{\mathcal{O}}_w\subseteq L_w$ and so on will be
used.

Put $\mathcal{X}=\mathbb{P}^1_{A}$.  Let
$\mathcal{X}_K=\mathbb{P}^1_K$ and $\mathcal{X}_s=\mathbb{P}^1_k$ be
respectively the generic and special fiber of $\mathcal{X}$ over
$A$. Let $\eta\in \mathcal{X}_s=\mathbb{P}^1_k$ denote the generic
point of $\mathcal{X}_s$. The valuation $v\in \Omega_A$ has a unique
center on the model $\mathcal{X}=\mathbb{P}^1_A$, which will be
denoted $P\in\mathcal{X}$. We have the following cases:

(1) $P\in \mathcal{X}_s=\mathbb{P}^1_k$, $P\neq 0,\,\infty\,,\eta$;

(2) $P=\eta\in \mathcal{X}_s=\mathbb{P}^1_k$;

(3) $P=\infty\in\mathcal{X}_s=\mathbb{P}^1_k$ or
$P=\infty\in\mathcal{X}_K=\mathbb{P}^1_K$;

(4) $P=0\in\mathcal{X}_s=\mathbb{P}^1_k$;

(5) $P$ is a closed point of $\mathbb{A}^1_K\subseteq
\mathcal{X}_K=\mathbb{P}^1_K$.

Our proof of Prop.$\;$\ref{prop1p4} will be a case-by-case argument,
which is divided into two parts with details in what follows.

\begin{proof}[Proof of Prop.$\;\ref{prop1p4}\,($Part I$)$]In the first
part of the proof, we treat cases (1)--(4).

\

Case (1). The valuation $v$ is centered at
$P\in\mathcal{X}_s\setminus\set{0\,,\,\infty\,,\,\eta}$.

In this case, we have $v(x)>0$ and $v(y)=0$. We may assume without
loss of generality that for some $0\le r_1\le r$, the numbers $n_i$
in \eqref{eq1p1} satisfy:
\begin{equation}\label{eq4p1new}
n_1=\cdots=n_{r_1}=0\quad \text{ and }\quad
n_{r_1+1}=\cdots=n_r=1\,.
\end{equation}Then $a_1\,,\dotsc, a_{r_1}$
and $a'_{r_1+1}=a_{r_1+1}/x\,,\dotsc, a'_r=a_r/x$ are units for $v$.
Let
\begin{equation}\label{eq4p2new}
q_1=\langle a_1\,,\dotsc,\,a_{r_1} \rangle\quad \text{ and }\quad
q_2=\langle a'_{r_1+1}\,,\dotsc, \, a'_r\rangle\,.
\end{equation}
Then $q=\langle a_1\,,\dotsc, a_r\rangle=q_1\bot\, x.q_2$ is
anisotropic only if $q_1$ and $q_2$ are both anisotropic. By
Springer's lemma (or Hensel's lemma), $q_i$ is anisotropic over
$F_v$ if and only if its residue form
$\ov{q}_i:=q_i\pmod{\mathfrak{m}_v}$ is anisotropic over
$\kappa(v)$. In the present situation, the two residue forms
$\ov{q}_i\,,\,i=1,\,2$ have coefficients in the subfield
$\kappa(P)\subseteq\kappa(v)$. Since $r\ge 5$, either $q_1$ or $q_2$
has rank $\ge 3$. Assume for example $q_1$ has rank $\ge 3$. The
residue field $\kappa(P)$ is a finite extension of $k$, so property
$(*)$ implies that $\ov{q}_1$ is isotropic over $\kappa(P)$ and a
fortiori over $\kappa(v)$. It follows that $q$ is isotropic over
$F_v$ as desired.

\

Case (2). The valuation $v$ is centered at the generic point $\eta$
of the special fiber $\mathcal{X}_s=\mathbb{P}^1_k$.

In this case, $v$ is the $x$-adic valuation on $A[y]$ and
$\kappa(v)=k(y)$. Let $w$ be the $x$-adic valuation on $A[\![y]\!]$,
so that $w|_{A[y]}=v|_{A[y]}$ and $\kappa(w)=k(\!(y)\!)$. Define
$q_1$ and $q_2$ as in \eqref{eq4p2new}. We have
\begin{equation}\label{eq1}
\begin{split}
\ov{q}_1&:=q_1\pmod{\mathfrak{m}_w}=\langle
\lambda_1y^{m_1}\,,\dotsc,
\lambda_{r_1}y^{m_{r_1}}\rangle\,,\\
\ov{q}_2&:=q_2\pmod{\mathfrak{m}_w}=\langle\lambda_{r_1+1}y^{m_{r_1+1}}\,,\dotsc,
\lambda_ry^{m_r}\rangle\,.
\end{split}
\end{equation}
Here we have identified each $\lambda_i\in\Sigma\subseteq A$ with
its canonical image in $k$. By hypothesis and Springer's lemma, we
may assume one of the two residue forms, say $\ov{q}_1$, is
isotropic over $k(\!(y)\!)$. By \eqref{eq1}, $\ov{q}_1$ has
coefficients in $k(y)$ and is isometric to $\mu_1\bot\,y.\mu_2$ over
$k(y)$ for some nonsingular quadratic forms $\mu_i$ over $k$.
Indeed, if $I$ (resp. $J$) denotes the subset of $\set{1\,,\dotsc,
r_1}$ consisting of indices $i$ such that $m_i$ is even (resp. odd),
then we may take $\mu_1=\langle\lambda_i\rangle_{i\in I}$ (resp.
$\mu_2=\langle\lambda_i\rangle_{i\in J}$). Applying Springer's lemma
to the form $\ov{q}_1/k(\!(y)\!)$ with respect to the discrete
valuation ring $k[\![y]\!]$, we conclude that either $\mu_1$ or
$\mu_2$ is isotropic over $k$. Then it is clear that $\ov{q}_1\cong
\mu_1\bot\,y.\mu_2$ is isotropic over $k(y)=\kappa(v)$. Since the
residue forms of $q$ mod $v$ coincide with those mod $w$, it follows
from Springer's lemma that $q$ is isotropic over $F_v$.

\

Case (3). The valuation $v$ is centered at
$P=\infty\in\mathcal{X}_s=\mathbb{P}^1_k$ or $P=\infty\in
\mathcal{X}_K=\mathbb{P}^1_K$.

In this case, we have $v(y)<0$ and $v(x)\ge 0$. Put $z=y^{-1}\in
F=K(y)$. We want to prove that $q$ is isotropic over $F_v$.

Recall that the coefficients of the diagonal form $q$ have the form
$a_i=\lambda_i.x^{n_i}.P_i$, where $\lambda_i\in \Sigma$,
$n_i\in\set{0\,,\,1}$ and $P_i$ is a distinguished polynomial in
$A[y]$ for each $i$. Let $m_i=\deg P_i$ be the degree of $P_i$ with
respect to the variable $y$. Then  in $F=K(y)$ we have
\[
P_i(y)=y^{m_i}(1+z.\rho_i)\,\quad\text{for some }\,\rho_i\in A[z]\,.
\]Set $b_i=\lambda_i.x^{n_i}.y^{m_i}\in F$ and let $q'/F$ be the
diagonal quadratic form $\langle b_1,\dotsc, b_r\rangle$. The two
forms $q=\langle a_i\rangle$ and $q'=\langle b_i\rangle$ are
isometric over $F_v$ since $1+z.\rho_i$ is a square in $F_v$ for
each $i$. So it suffices to prove the isotropy  over $F_v$ of the
form $q'=\langle b_i\rangle$.

We may assume the numbers $n_i$ are given as in \eqref{eq4p1new}, so
that $q'=q'_1\bot x.q'_2$ with
\[
q'_1=\langle \lambda_1y^{m_1}\,,\dotsc,
\lambda_{r_1}y^{m_{r_1}}\rangle\;,\quad q'_2=\langle
\lambda_{r_1+1}y^{m_{r_1+1}}\,,\dotsc, \lambda_ry^{m_r}\rangle\,.
\]There are diagonal quadratic forms $\mu_j\,,\,j=1,\dotsc, 4$, where
$\mu_1\,,\,\mu_2$ have coefficients in $\set{\lambda_1\,,\dotsc,
\lambda_{r_1}}\subseteq \Sigma$ and $\mu_3\,,\,\mu_4$ have
coefficients in $\set{\lambda_{r_1+1}\,,\dotsc, \lambda_r}\subseteq
\Sigma$, such that $q'_1\cong \mu_1\bot y.\mu_2$ and $q'_2\cong
\mu_3\bot y.\mu_4$ over $F=K(y)$. Observe that the two residue forms
of $q$ with respect to the $x$-adic valuation on $F$ are isometric
to the forms $\mu_1\bot y.\mu_2$ and $\mu_3\bot y.\mu_4$. A close
inspection of the above proof for case (2) shows that not all of the
four forms $\mu_j$ are anisotropic over $k$. Since
\[
q'\cong \mu_1\bot y.\mu_2\bot x.(\mu_3\bot y.\mu_4)\;\quad\text{over
}\; F=K(y)\,,
\]it follows easily that $q'$ is isotropic over $F_v$, whence the
isotropy of $q$ over $F_v$.

\

Case (4). The valuation $v$ is centered at the origin
$P=0\in\mathbb{P}^1_k$ of the special fiber.

By the definition of the set $\Omega_A$, the valuation
$v\in\Omega_A$ corresponds to a codimension 1 point $p$ of a regular
proper model $\mathcal{P}/A$ of $\mathbb{P}^1_K$. Since the center
of $v$ on $\mathcal{X}$ lies in the special fiber, $v(x)>0$. The
point $p\in\mathcal{P}$ lies in the special fiber of $\mathcal{P}/A$
since otherwise the valuation $v$ must be trivial on
$K=\mathrm{Frac}(A)$. The residue field $\kappa(v)$ is then the
function field of a curve over $k$. So we have
\[
\mathrm{trdeg}_k\kappa(v)=1=\dim\mathbb{P}^1_A-1\,.
\]

By Lemma$\;$\ref{lemma2p1}, there is a scheme $\mathcal{X}_n\to
\mathcal{X}=\mathbb{P}^1_A$ obtained by a sequence of blow-ups at
closed points lying over $0\in \mathcal{X}_s=\mathbb{P}^1_k$ such
that $\mathcal{O}_v=\mathscr{O}_{\mathcal{X}_n\,,\,x_n}\subseteq F$
for some codimension 1 point $x_n\in \mathcal{X}_n$. If we consider
the same sequence of blow-ups which is carried out on $\Spec
A[\![y]\!]$ this time, then we get a discrete valuation
$w\in\Omega_R$ of $L$ which extends $v$. Now we have inclusions
$A[y]\subseteq \mathcal{O}_v\subseteq \mathcal{O}_w$ and
$\kappa(v)=\kappa(w)$. Let $q_1\,,\,q_2$ be diagonal quadratic forms
with coefficients in $\wh{\mathcal{O}}_v^*$ such that
\[
q\cong q_1\bot\; \pi_v.q_2\quad\text{ over }\;\;F_v\,.
\]Since $q$ is isotropic over $L_w$ by assumption, applying
Springer's lemma to $w$ shows that $\ov{q}_1=q_1\pmod{\mathfrak{m}_v}$
or $\ov{q}_2=q_2\pmod{\mathfrak{m}_v}$ has a nontrivial zero in
$\kappa(w)=\kappa(v)$. One more application of Springer's lemma,
with respect to $v$ this time, proves that $q$ is isotropic over
$F_v$.
\end{proof}

\section{End of the proof}
To prove Prop.$\;$\ref{prop1p4} in case (5), we need the following
form of the Weierstra{\ss} preparation theorem.

\begin{lemma}[Weierstra{\ss}]\label{lemma5p1}
Let $A$ be a complete discrete valuation ring and $A[\![y]\!]$ the
ring of formal power series in one variable over $A$. Let $P\in
A[y]$ be a distinguished polynomial and $f\in A[y]$.

$(\mathrm{i})$ For any $g\in A[\![y]\!]$, there is a unique
expression
\[
g=Q.P+R
\]where $Q\in A[\![y]\!]$ and $R\in A[y]$ is a polynomial of degree $\le \deg
P-1$. In particular,
\[
A[y]/(P)\cong A[\![y]\!]/(P)\,.
\]

$(\mathrm{ii})$ If $f$ divides $P$ in $A[y]$, then there is a unit
$u$ in $A$ such that $uf$ is a distinguished polynomial.
\end{lemma}
\begin{proof}
(i) See e.g. \cite[p.114, Prop.$\;$7.4]{Wash}. Note that the
isomorphism $A[y]/(P)\cong A[\![y]\!]/(P)$ implies that $P$ is
irreducible in $A[y]$ if and only if $P$ is irreducible in
$A[\![y]\!]$ and that $P$ divides a polynomial $f$ in $A[y]$ if and
only if $P$ divides $f$ in $A[\![y]\!]$.

(ii) Assume $P=fg$ with $g\in A[y]$. The hypothesis implies
that the coefficient $a_0$ of $y^{\deg f}$ in $f$ is a unit in $A$ since $P$ is a monic polynomial. Let $k$ be the residue field of $A$ and let $A[y]\to k[y]\,,\;F\mapsto \overline{F}$ denote the canonical reduction map.
 By considering the factorization $y^{\deg P}=\overline{P}=\overline{f}\cdot\overline{g}$ in $k[y]$, 
 we see that $u:=a_0^{-1}\in A^*$ has the required property.
\end{proof}

\begin{proof}[Proof of Prop.$\;\ref{prop1p4}\,($Part II$)$]
We now consider the only remaining case, case (5). This is the case
where the center $P$ of the valuation $v$  lies in
$\mathbb{A}^1_K\subseteq \mathcal{X}_K=\mathbb{P}^1_K$.

We have $\mathscr{O}_{\mathcal{X},\,P}=\mathcal{O}_v$ since the two
rings are both discrete valuation rings with fraction field $F$. So
$v$ is defined by an irreducible polynomial $f\in A[y]$ with $x\nmid
f$.

If none of the polynomials $P_i,\,i=1,\dotsc, r$ is divisible by
$f$, then $q$ has coefficients in
$\mathcal{O}_v^*=\mathscr{O}_{\mathcal{X},\,P}^*$. Now the residue
field $\kappa(v)=\kappa(P)$ is a finite extension of $K$ and the
residue form $\ov{q}=q\pmod{\mathfrak{m}_v}$ has rank $r\ge 5$. By
property $(**)$ (cf. Remark$\;$\ref{remark1p3}), $\ov{q}$ is
isotropic over $\kappa(v)$. It follows from Springer's lemma (or
Hensel's lemma) that $q$ is isotropic over $F_v$.

Assume next $f$ divides some $P_i$, say $f\,|\,P_1$. By
Lemma$\;$\ref{lemma5p1}, multiplying $f$ by a unit in $A$ if
necessary, we may assume that $f$ is an irreducible distinguished
polynomial. In $A[\![y]\!]$, $f$ is still an irreducible element.
The $f$-adic valuation on $R=A[\![y]\!]$ determines a discrete
valuation $w\in \Omega_R$ which extends $v\in\Omega_A$. We have
\[
\kappa(v)=\mathrm{Frac}(A[y]/(f))=\mathrm{Frac}(A[\![y]\!]/(f))=\kappa(w)
\]and $F_v\subseteq L_w$. Using the argument with the first and second residue forms and
Springer's lemma, we conclude as in case (4) that $q$ is isotropic
over $F_v$.
\end{proof}

We are now ready to give the proof of Thm.$\;$\ref{thm1p2}.

\begin{proof}[Proof of Thm.$\;\ref{thm1p2}$]
Let $q$ be any quadratic form of rank $r\ge 5$ over
$L=\mathrm{Frac}(R)$ and assume that $q$ is isotropic over $L_w$ for
every $w\in\Omega_R$. Without loss of generality, we may assume
$q=\langle a_1,\dotsc, a_r\rangle$ for some nonzero elements $a_i\in
R=A[\![y]\!]$. By the usual form of the Weierstra{\ss} preparation
theorem (see e.g. \cite[p.115, Thm.$\;$7.3]{Wash}), each $a_i$ may
be written as
\[
a_i=x^{n_i}.P_i.U_i\quad \text{ with }\;n_i\in\mathbb{N}\,,\,U_i\in
R^*\;\text{ and }\; P_i\;\text{ a distinguished polynomial in
}\,A[y]\,.
\]For any power series $f=\sum^{\infty}_{i=0}a_iy^i\in R=A[\![y]\!]$ which is invertible in
$R$, letting $\lambda\in\Sigma$ be the unique element such that
$\lambda^{-1} a_0\equiv 1\pmod{xA}$, we have
\[
\lambda^{-1}f\equiv 1\;\pmod{\mathfrak{m}_R}\,.
\]Since $R$ is complete, it follows that $\lambda^{-1} f$ is a square in
$R$. So after scaling out squares we may assume that the
coefficients $a_i$ have the form described in
Prop.$\;$\ref{prop1p4}. Now the quadratic form $q$ is defined over
$F=K(y)$ and by Prop.$\;$\ref{prop1p4}, it is isotropic over $F_v$
for every $v\in\Omega_A$. The local-global principle with respect to
discrete valuations in $\Omega_A$ is proved for quadratic forms of
rank $\ge 3$ in \cite[Thm.$\;$3.1 and Remark$\;$3.2]{CTPaSu}. Hence,
$q$ is isotropic over $F$ and a fortiori over $L$.
\end{proof}

\begin{remark}\label{remark5p2}
In Thm.$\;$\ref{thm1p2}, assume that $A=k[\![x]\!]$ with $k$ a $C_1$
field of characteristic $\neq 2$ or $A=\mathcal{O}_K$ with $K$ a
$p$-adic number field ($p$ an odd prime). Then every quadratic form
of rank $\ge 9$ is isotropic over $F=K(y)$. In the former case, it
is well-known that $F=k(\!(x)\!)(y)$ is a $C_3$ field. For the case
$A=\mathcal{O}_K$, this statement is firstly proved by Parimala and
Suresh \cite{PaSu}, and then two more recent proofs using different
methods are given in \cite[Coro.$\;$4.15]{HHK} and
\cite[Coro.$\;$3.4]{CTPaSu} as consequences of their main theorems.
Still another proof (including the case $p=2$), which builds upon
the work of Heath-Brown \cite{HB10}, has been announced by Leep
\cite{Le10}.

An easy argument using the Weierstra{\ss} preparation theorem shows
that every quadratic form of rank $\ge 9$ is isotropic over
$L=\mathrm{Frac}(A[\![y]\!])$. So in these cases, the local-global
principle in Thm.$\;$\ref{thm1p2} is only interesting for quadratic
forms of rank $5\le r\le 8$.
\end{remark}

\

\noindent \emph{Acknowledgements.} The author thanks his advisor,
Prof.\! Jean-Louis Colliot-Th\'{e}l\`{e}ne, for many valuable
discussions and comments. Thanks also go to Prof.\! Raman Parimala,
who has read the manuscript carefully and given comments that led
the author to find that the earlier version of
Theorem$\;$\ref{thm1p2} may be generalized to the present one.
The author is grateful to the referee for useful comments.

\end{document}